\documentclass[12pt,a4paper]{amsart}

\usepackage{latexsym}
\usepackage{amssymb}
\usepackage{amsmath}
\usepackage{amsthm}

\makeatletter
\def\@settitle{\begin{flushleft}%
  \baselineskip14\p@\relax
    \normalfont\Large\bf
  \@title
  \end{flushleft}%
}
\def\section{\@startsection{section}{1}%
  \z@{.7\linespacing\@plus\linespacing}{.5\linespacing}%
  {\normalfont\bf}}

\def\@setauthors{%
  \begingroup
  \def\thanks{\protect\thanks@warning}%
  \trivlist
  \footnotesize \@topsep30\p@\relax
  \advance\@topsep by -\baselineskip
  \item\relax
  \author@andify\authors
  \def\\{\protect\linebreak}%
  \larger\sc\authors%
  \ifx\@empty\contribs
  \else
    ,\penalty-3 \space \@setcontribs
    \@closetoccontribs
  \fi
  \endtrivlist
  \endgroup
}

\makeatother

\usepackage{exscale}
\usepackage[english]{babel}
\usepackage{verbatim}
\usepackage{fullpage}
\usepackage{fix-cm}
\usepackage{enumitem}
\usepackage[colorlinks, breaklinks, allcolors=blue]{hyperref} 

\swapnumbers
\theoremstyle{plain}
\newtheorem{theorem}{Theorem}[section]
\newtheorem{lemma}[theorem]{Lemma}
\newtheorem{proposition}[theorem]{Proposition}
\newtheorem{corollary}[theorem]{Corollary}

\theoremstyle{definition}

\newtheorem*{notation}{Notation}

\newcommand{\codim}{\operatorname{codim}}
\newcommand{\Aut}{\operatorname{Aut}}
\newcommand{\PGL}{\operatorname{PGL}}

\newcommand{\Char}{\operatorname{char}\,}
\newcommand{\Chi}{\mathcal X}

\newcommand{\Sym}{\operatorname{Sym}}

\newcommand{\SL}{\operatorname{SL}}

\newcommand{\SO}{\operatorname{SO}}

\newcommand{\GG}{\mathbb{G}}
\newcommand{\PP}{\mathbb{P}}

\newcommand{\Wt}{\widetilde{W}}
\newcommand{\Bnn}{\mathfrak{B}_{00}}
\newcommand{\Bn}{\mathfrak{B}_{0}}
\newcommand{\B}{\mathfrak{B}}

\newcommand{\sA}{{\mathsf A}}
\newcommand{\sB}{{\mathsf B}}
\newcommand{\sG}{{\mathsf G}}

\usepackage{parskip}

\usepackage{microtype}

\begin{document}
\title{On the {\boldmath$W$}-action on {\boldmath$B$}-sheets in positive characteristic}
\author{Friedrich Knop, Guido Pezzini\newline FAU Erlangen-N\"urnberg}
\address{FAU Erlangen-N\"urnberg, Dept. Mathematik, Cauerstr. 11, 91058 Erlangen, Germany}

\begin{abstract}

  Let $G$ be a connected reductive group defined over an algebraically
  closed base field of characteristic $p\ge0$, let $B\subseteq G$ be
  a Borel subgroup, and let $X$ be a $G$-variety. We denote the
  (finite) set of closed $B$-invariant irreducible subvarieties of $X$ that are
  of maximal complexity by $\Bn(X)$. The first named author has
  shown that for $p=0$ there is a natural action of the Weyl group $W$
  on $\Bn(X)$ and conjectured that the same construction yields a
  $W$-action whenever $p\ne2$. In the present paper, we prove this
  conjecture.

\end{abstract}

\maketitle

\section{Introduction}

Let $G$ be a connected reductive group defined over an algebraically closed
base field $\Bbbk$ with Borel subgroup $B\subseteq G$. For any
$G$-variety $X$ let $\B(X)$ be the set of all $B$-stable irreducible closed
subvarieties
of $X$. The \emph{complexity} $c(Y)$ of $Y\in\B(X)$ is the codimension
of a generic $B$-orbit in $Y$ or, equivalently, the transcendence
degree of $\Bbbk(Y)^B$. It is a result of Vinberg \cite{Vi86} that the
complexity takes its maximal value for $Y=X$. Of particular interest
is therefore the subset
\[
\Bn(X):=\{Y\in\B(X)\mid c(Y)=c(X)\}.
\]
This set contains $X$ and is finite since it consists of the closures
of all $B$-sheets with a maximal number of parameters
(see \cite[Proposition 4.1]{Kn95}). The most
important case is that of a spherical variety (i.e. $c(X)=0$) when
$\Bn(X)=\B(X)$ is just the set of all $B$-orbit closures.

Let $W$ be the Weyl group of $G$. In \cite{Kn95}, an action of $W$ on
$\Bn(X)$ was constructed whenever the base field has characteristic
zero. On the other side, in the same paper an example was given
showing that the construction does not work in characteristic $2$.  In
any other characteristic, the situation was unclear so far. The
purpose of this paper is to close this gap by showing that the method
of \cite{Kn95} does indeed define a $W$-action on $\Bn(X)$ in every
characteristic $\ne2$.

It was already indicated \cite{Kn95} that the problem can be reduced
to the following special case: the characteristic $p$ of $\Bbbk$ is
$>2$, the group $G$ is semisimple of rank $2$, and the variety $X$ is
of the form $X=G/H$ where $H$ is a connected non-spherical subgroup of
$G$. Moreover, we may replace $\Bn(X)$ by a certain subset $\Bnn(X)$
(see \S\ref{s:action} for its definition).

It is then enough to consider only those $H$ for which $\Bnn(G/H)$
consists of more than one element. Remarkably, under such assumptions
the proof can be completed by showing that there exist a
spherical subgroup $K\subseteq G$ and a bijection $\Bnn(G/H)\to
\Bnn(G/K)$ that is compatible with the operation of the simple
reflections of $W$.

While several subgroups require case-by-case considerations, others
can be treated with general arguments, e.g.\ solvable subgroups.

\begin{notation}
  All varieties are defined over an algebraically closed field $\Bbbk$
  of characteristic $p\ge0$. We denote by $G$ a connected reductive
  group, we fix a Borel subgroup $B\subseteq G$, whose unipotent
  radical is denoted by $U$, and a maximal torus $T\subseteq
  B$. The opposite Borel subgroup with respect to $T$, and its
  unipotent radical, are denoted by $B^{-}$ and $U^{-}$, respectively.
  
  Denote by $R$ the set of roots with respect to $T$, by $R^{+}$ the
  set of positive roots corresponding to $B$, and by $S\subseteq R$
  the set of simple roots.
  If $G$ is simple, then
  the simple roots $\alpha_1,\alpha_2,\ldots$ and the fundamental
  dominant weights $\omega_1,\omega_2,\ldots$ will be
  numbered as in \cite[Planches I--IX]{Bou}. The $1$-dimensional unipotent subgroup of
  $G$ associated to a root $\gamma$ is
  denoted by $U_\gamma$, and we choose once and for all an isomorphism
  $u_\gamma:\GG_a\to U_\gamma$. The Weyl group of $G$ is denoted by $W$,  
  its longest element is denoted by $w_0$,  and the simple reflection
  associated to $\alpha\in S$ is denoted by $s_\alpha$.

If $\omega$ is a dominant weight, then $V_G(\omega)$ denotes the
irreducible $G$-module of highest weight $\omega$. If no confusion arises,
we simply write $V(\omega)$.

  If  $\alpha$ is a simple root of $G$ then $P_\alpha$ denotes the
  minimal parabolic subgroup of $G$ which is generated by $B$ and $U_{-\alpha}$,
  and $\pi_\alpha\colon G/B\to G/P_\alpha$ the natural map
  $gB\mapsto gP_\alpha$. If $g\in G$ and $H\subseteq G$, we use the
  notation ${}^gH= gHg^{-1}$. For any algebraic group $H$, we denote
  by $H'$   its commutator subgroup, by $H^r$ (resp.\ $H^u$) its
  radical (resp.\ unipotent radical), and by $\Chi(H)$ its group of
  characters,
  i.e.\ the set of all algebraic group homomorphisms $H\to \GG_m$.
\end{notation}
\section{The action of the Weyl group}\label{s:action}

We recall some definitions and facts from \cite{Kn95}. Let $X$ be an
algebraic variety equipped with a $G$-action. To avoid
confusion, a $B$-stable irreducible closed subvariety $Z$ of $X$
is denoted sometimes by $(Z)$ if we are referring to it as an
element of $\B(X)$.

We define the {\em
character group} $\Chi(Z)$ of $Z$ as the group of $B$-eigenvalues of
$B$-eigenvectors in $\Bbbk(Z)$. The {\em rank} $r(Z)$ of $Z$ is the
rank of the free abelian group $\Chi(Z)$.

We also define the following subset of $\Bn(X)$:
\[
\begin{array}{ccl}
\Bnn(X) & = & \left\{ (Z) \in \B(X) \; \middle\vert\;
c(Z) = c(X),\; r(Z) = r(X) \right\}.
\end{array}
\]

If $X=G/H$ is homogeneous, then there is a canonical bijection between
$\B(X)$ and the set of $H$-stable $H$-irreducible closed subsets of
$G/B$. Here ``$H$-irreducible'' means that $H$ acts transitively on
the set of irreducible components. Throughout the paper we will
sometimes implicitly make use of this bijection.

Let $\Wt$ be the group defined by generators $s_\alpha, \alpha\in S$ and relations $s_\alpha^{2}=e, \alpha\in S$.
In \cite{Kn95} an action of
$\Wt$ of $\Bn(X)$ has been defined as follows. Let $\alpha$ be a simple
root, and recall that $P_\alpha\supset
B$ is the minimal parabolic subgroup of $G$ corresponding to
$\alpha$. Then $Z\mapsto P_\alpha Z$ is an idempotent selfmap of
$\Bn(X)$. Its image $\Bn^\alpha(X)$ consists of those $Z\in\Bn(X)$
which are $P_\alpha$-stable. Thus, for $Y\in\Bn^\alpha(X)$ the fibers

\[
\Bn(Y, P_\alpha) = \left\{ Z \in \Bn(X) \; \middle\vert\; P_\alpha Z = Y
\right\}
\]
form partition of $\Bn(X)$. Now, the element $s_\alpha$ acts on
each block $\Bn(Y, P_\alpha)$ as an involution according to the
following table:
\[
\begin{array}{clll}
\textrm{Type} & 
\Bn(Y,P_\alpha)  && s_\alpha\textrm{--action}\\
\noalign{\smallskip}\hline\noalign{\smallskip}
(G) & \{Y \}  && s_\alpha \cdot (Y) = (Y) \\
(U) & \{Y, Z \}, \; &r(Z)=r(Y)  & s_\alpha \cdot (Y) = (Z),\ 
s_\alpha \cdot (Z) = (Y) \\
(N) & \{Y, Z \}, \; &r(Z) < r(Y)  & s_\alpha \cdot (Y) = (Y), \ 
s_\alpha \cdot (Z) = (Z) \\
(T) & \{Y, Z_0,Z_\infty \},&r(Z_0) =r(Z_\infty) < r(Y)  &
s_\alpha \cdot (Y) = (Y),\ s_\alpha \cdot (Z_0) = (Z_\infty)\\
&& & s_\alpha \cdot (Z_\infty) = (Z_0)
\end{array}
\]
No other cases can occur (see \cite[\S4]{Kn95}). This definition is based on a construction of Lusztig and Vogan in the case of symmetric spaces (see \cite[\S3]{LV83}). Under the hypothesis that $B$ has a dense orbit on $G/H$ (in which case $G/H$ is called a {\em spherical} homogeneous space, and $H$ a {\em spherical} subgroup of $G$), the link is explained in more details in \cite[\S1 and \S5]{Kn95}.

Let us further analyze the four above cases in case that $X=G/H$ is a homogeneous variety.

Suppose that $Y\in \Bn(G/H)$ is $P_\alpha$-stable. Then, considered
as an $H$-stable subset of $G/B$, it satisfies $\pi^{-1}(\pi(Y))=Y$
where $\pi= \pi_\alpha\colon G/B\to G/P_\alpha$.

The fiber $\pi^{-1}(x)$ over any $x\in\pi(Y)$ is isomorphic to $\PP^1$
and equipped with a natural transitive action of $G_x$; we fix the
isomorphism $\pi^{-1}(x)\cong\PP^1$ and the corresponding homomorphism
$\Phi\colon G_x\twoheadrightarrow\PGL(2)=\Aut\PP^1$. The fiber $\pi^{-1}(x)$
intersects each $Z\in\Bn(Y,P_\alpha)$ in an $H_x$-stable subset. If
$x$ is general in $\pi(Y)$ then the four cases of the above table
correspond resp.\ to the following:

\begin{itemize}

\item[$(G)$]$\Phi(H_x)$ is $\PGL(2)$ or finite,

\item [$(U)$]the unipotent radical of $\Phi(H_x)$ is a maximal unipotent subgroup of
$\PGL(2)$,

\item [$(N)$]$\Phi(H_x)$ has trivial unipotent radical and
$|\Bn(Y,P_\alpha)| = 2$,

\item [$(T)$]$\Phi(H_x)$ is a maximal torus of $\PGL(2)$.

\end{itemize}

In \cite{Kn95} it is shown that the $\Wt$-action on $\Bn(G/H)$ descends
to an action of the Weyl group $W$ of $G$ under several different
assumptions, see \cite[Theorem 4.2]{Kn95}. In particular, this is true
if $H$ is a spherical subgroup of $G$ and $\Char\Bbbk \neq 2$.

We come to our main result.

\begin{theorem} \label{thm:waction}
Let $\Char \Bbbk >2$ and let $G$ be semisimple of rank $2$. For
any connected subgroup $H\subseteq G$, the $\Wt$-action defined
above induces a $W$-action on $\Bnn(G/H)$.
\end{theorem}
Sections \ref{s:solvable}--\ref{s:other} are devoted to the proof.
Precisely, the theorem follows from Corollaries~\ref{cor:solvable},
\ref{cor:red}, \ref{cor:levi} and \ref{cor:other}.

Thanks to \cite[\S7]{Kn95}, the above theorem implies the following.

\begin{corollary}
Let $\Char \Bbbk \ne2$. Then for any $G$-variety $X$ the
$\Wt$-action defined above induces a $W$-action on $\Bn(X)$.
\end{corollary}

Before going into the details of the proof of Theorem~\ref{thm:waction},
we report two general results on $\Bnn(G/H)$.

\begin{proposition}\label{prop:transitive}
  Let $X=G/H$ be a homogeneous $G$-variety and $Z\in\Bnn(X)$ with
  $Z\ne X$. Then there is $\alpha\in S$ such that $\dim
  s_\alpha\cdot(Z)=\dim Z+1$. In particular:
\begin{itemize}
\item The $\Wt$-action on $\Bnn(X)$ is transitive.
\item If $|\Bnn(X)|>1$ then there is $Z\in\Bnn(X)$ with $\codim_XZ=1$.
\end{itemize}
\end{proposition}

\begin{proof}
  Suppose $\dim s_\alpha\cdot(Z)\le\dim Z$. Then the definition of the
  $s_\alpha$-action and the fact that $Z$ is of maximal rank implies
  $P_\alpha Z=Z$. Since $Z\ne X$ and since $X$ is homogeneous there
  is $\alpha\in S$ with $P_\alpha Z\ne Z$. Then $\dim
  s_\alpha\cdot(Z)=\dim Z+1$ since $s_\alpha$ increases dimension by
  at most $1$.

  This implies that the action is transitive since any $Z$ can be
  moved in finitely many steps to $(X)$. Finally, starting from any
  $Z\ne X$ the next to the last step will be of codimension one.
\end{proof}

\begin{lemma}\label{lemma:normtorus}
Let $H\subseteq K\subseteq G$ be subgroups such that $H$ is normal in
$K$ and the quotient $K/H$ is diagonalizable. Let $\pi\colon G/H\to
G/K$ be the standard projection. Then the map $Y\mapsto \pi^{-1}(Y)$ is
a $\Wt$-equivariant bijection $\Bnn(G/K)\to\Bnn(G/H)$.
\end{lemma}
\begin{proof}
We show by induction on $\codim_{G/K}Z$ that $\pi^{-1}(Z) \in
\Bnn(G/H)$ for all $Z\in \Bnn(G/K)$. This is true if $Z=G/K$, hence we can suppose
$Z\subsetneq G/K$.

Let $Y\in \Bnn(G/K)$ and $\alpha\in S$ be such that $\dim Y=\dim Z+1$
and $s_\alpha\cdot (Z)=(Y)$; notice that such elements exist thanks to
Proposition~\ref{prop:transitive}. In particular, $Z$ is the unique
element of $\Bnn(G/K)$ such that $P_\alpha Z = Y$. 

By induction $Y'=\pi^{-1}(Y)$ is in $\Bnn(G/H)$. Then for a general $y
\in Y$ the image of the stabilizer $(P_\alpha)_y$ into
$\Aut(P_\alpha/B) = \Aut(\PP^1) = \PGL(2)$ is a proper subgroup that
contains a maximal unipotent subgroup.

Let $y' = gH$ be a general point of $Y'$: its stabilizer
$(P_\alpha)_{y'}$ is equal to $P_\alpha\cap {}^gH$, and we may compare
the latter with $(P_\alpha)_y = P_\alpha\cap {}^gK$ where $y = gK =
\pi(y')$. We have that $P_\alpha\cap{}^gH$ is normal in
$P_\alpha\cap{}^gK$, and the quotient is a diagonalizable group.
Therefore the image of $(P_\alpha)_{y'}$ into $\PGL(2)$ is again proper
and contains a maximal unipotent subgroup. Then the $P_\alpha$-orbit of
$y'$ in this case is the union of two $B$-orbits, and there exists a
unique $Z'\in\B(G/H)$ such that $P_\alpha Z' = Y'$. It follows that
$P_\alpha \pi^{-1}(Z) = Y'$, so $Z'=\pi^{-1}(Z)$, the latter lies in
$\Bnn(G/H)$, and $s_\alpha\cdot(Y')=(Z')$.

Denote with $f\colon \Bnn(G/K)\to \Bnn(G/H)$ the map $Y\mapsto \pi^{-1}(Y)$.
We have shown that $f$ is injective, and that whenever $s_\alpha$ exchanges
two elements $Y,Z\in \Bnn(G/K)$ then it exchanges $f(Y)$ and
$f(Z)$.

To show that $f$ is $\Wt$-equivariant, which also implies the surjectivity
of $f$ since both sets are $\Wt$-orbits, it remains only to consider
$Y\in\Bnn(G/K)$ and $\alpha\in S$ such that $Y$ is $P_\alpha$-stable
and $s_\alpha\cdot(Y)=(Y)$, and to check that $s_\alpha\cdot(f(Y))=(f(Y))$.
In this case, the image of $P_\alpha\cap {}^gK$
in $\PGL(2)$ is either the whole $\PGL(2)$, a maximal torus, the
normalizer of a maximal torus, or a finite group. Since
$P_\alpha\cap{}^gH$ is a normal subgroup of $P_\alpha\cap {}^gK$ such
that the quotient is diagonalizable, we conclude that its image in
$\PGL(2)$ also belongs to one of the above four possible types of
subgroups. Hence $s_\alpha\cdot(f(Y))=(f(Y))$.
\end{proof}

\section{Solvable subgroups} \label{s:solvable}

\emph{We assume from now on that $p:=\Char\Bbbk>2$.}

\begin{lemma}\label{lemma:U}
  Suppose that $G$ is semisimple of rank $2$, and let $\alpha, \beta$
    be the two simple roots. Let $N$ be the normal subgroup of $U$
    generated by $U_\alpha$. Then $N=P_\beta^u$.
\end{lemma}

\begin{proof}
  Clearly, $N\subseteq P_\beta^u$. To show the reverse inclusion we
  recall that 
  \begin{equation}\label{Equation:commutator}
  u_\beta(y)^{-1}u_\alpha(x)u_\beta(y)=
  \prod_{\gamma\in R^+\setminus\{\beta\}}u_\gamma(f_\gamma(x,y))\in N
  \end{equation}
where each $f_\gamma(x,y)$ is a non-constant polynomial (see
\cite[Expos\'e XXIII: Proposition 3.1.2(iii), Proposition 3.2.1(iii), Proposition
3.3.1(iii), Proposition 3.4.1(iii)]{SGA3} for groups of type \
$\mathsf A_1\times \mathsf A_1$, $\mathsf A_2$, $\mathsf B_2$ and
$\mathsf G_2$, respectively). Since $N$ is normalized by $T$ we have
$N=\prod_{\gamma\in R'}U_\gamma$ where $R'\subseteq R^+$ is a
subset. Thus, all factors of the right hand side of
(\ref{Equation:commutator}) are also in $N$. This shows
$R^+\setminus\{\beta\}\subseteq R'$ and therefore
$P_u^\beta\subseteq N$ , as claimed.
\end{proof}

\begin{proposition}\label{prop:solvable}
  Let $G$ be semisimple of rank $2$ and $H\subseteq G$ be a connected
  solvable subgroup. Then $|\Bnn(G/H)|=1$, or $H$ contains the
  unipotent radical of a parabolic subgroup of $G$.
\end{proposition}
\begin{proof}
  Suppose $|\Bnn(G/H)|> 1$. Then there is a simple root
  $\alpha$ and $Z\in\Bnn(G/H)$ with $Z\ne G/H$ and $s_\alpha(Z)=(G/H)$
  (Proposition \ref{prop:transitive}). Let $Z'\subset G/B$ correspond to
  $Z$. The definition of the
  $s_\alpha$ -action implies that $\pi_\alpha(Z')=G/P_\alpha$ and that
  the generic isotropy group of $H$ in $Z'$, hence in $G/P_\alpha$,
  contains a subgroup isomorphic to $\GG_a$. Without loss of
  generality we may assume that $H\subseteq B^-$. The $U^-$-orbit of
  $e_GP_\alpha$ is dense in $G/P_\alpha$ and, for $u\in U^-$, the
  isotropy group of $x=uP_\alpha$ in $B^-$ is
  $B^-_x={}^u(TU_{-\alpha})={}^uT\,{}^uU_{-\alpha}$. Because of
  $H_x\subseteq B^-_x$ we conclude that ${}^uU_{-\alpha}\subseteq H$
  for $u\in U^-$ general and therefore, by continuity, for all $u\in
  U^-$. Then Lemma \ref{lemma:U} implies $H\supseteq P_{-\beta}^u$
  where $\beta$ is the simple root different from $\alpha$.
\end{proof}

\begin{corollary} \label{cor:solvable} Theorem \ref{thm:waction} holds
  for all connected solvable subgroups $H\subset G$.
\end{corollary}
\begin{proof}
  If $|\Bnn(G/H)|=1$ there is nothing to prove. Otherwise, without
  loss of generality, the unipotent radical $H^u$ is either $U$ or
  $P_\alpha^u$ for some simple root $\alpha$ (Proposition
  \ref{prop:solvable}). The claim is also known to be true if $H$ is
  spherical. This leaves to check only
  $H=\widetilde T P_\alpha^u$ where either $\widetilde T=\{e\}$ or
    $\widetilde T = (\ker\alpha)^\circ\subset T$. In either case
    Lemma~\ref{lemma:normtorus} applies to $H$ and $K=TP_\alpha^{u}$;
    since $TP_\alpha^{u}$ is spherical, the corollary follows.
\end{proof}

\section{Reductive subgroups}  \label{s:reductive}

In this section, we consider the case where $H$ is a reductive
subgroups of $G$. Clearly, by the preceding section we may assume that
$H$ is not a torus. So, its semisimple rank is one or two.

\begin{proposition}
Let $G$ be a semisimple group of rank $2$ and $H\subseteq G$ a
connected reductive subgroup of semisimple rank $2$. Then $H$ is
spherical.
\end{proposition}
\begin{proof}
The root system of $H$ is given by a subroot system
of rank $2$ of the root system of $G$. These are easily determined: in
case $G=\sA_1\times\sA_1$
or $G=\sA_2$ then only $H=G$ is possible. If $G=\sB_2$ then there is 
additionally $H=\sA_1\times\sA_1$. Finally, if $G=\sG_2$ there is an
extra-complication if $p=3$. In general, $H=\sA_1\times\sA_1$ and
$H=\sA_2$ is possible, the latter corresponding to the set of long
roots. If $p=3$ then there is another subgroup of type $\sA_2$
corresponding to the short roots. But that is mapped by a special
isogeny (see \cite[\S3.3]{BT73}) to the former. We conclude that $H$
is a spherical subgroup in every case thanks to
\cite[Theorem 4.3]{Br98}.
\end{proof}

We are left with the case where $H$ has semisimple rank one.

\begin{lemma}\label{lemma:UH-fixed}

Let $G$ be a semisimple group of rank $2$ and $H\subset G$
  a connected reductive subgroup with $\operatorname{rk}_{\mathrm{ss}}H=1$
  and $|\Bnn(G/H)|>1$. Let $U_H\cong \mathbb G_a$ be a maximal
  connected
  unipotent subgroup of $H$. Then there is a simple root $\alpha\in S$ such
  that
  the fixed point set $(G/P_\alpha)^{U_H}$ has a component of
codimension at most one in $G/P_\alpha$.

\end{lemma}

\begin{proof}

  Arguing as in the proof of Proposition \ref{prop:solvable}, there is
  a simple root $\alpha$ such that the isotropy group $H_y$ is not
  reductive for
  general $y\in G/P_\alpha$. Then ${\rm rk}_{\rm ss}H=1$ implies that $H_y$
  contains a conjugate of $U_H$ for general $y\in G/P_\alpha$ or, in
  other words, a general $H$-orbit in $G/P_\alpha$ contains an
  $U_H$-fixed point. The normalizer of $U_H$ in $H$ is of codimension
  one. Thus, the $U_H$-fixed points are of codimension at most one in
  any general orbit $Hy$. We conclude that $(G/P_\alpha)^{U_H}$ has a
  component $C$ which is of codimension at most one in $G/P_\alpha$.

\end{proof}

Now we analyze the situation of the preceding lemma further.

\begin{lemma}\label{lemma:red2}

  Let $G,H,U_H$ be as in Lemma~\ref{lemma:UH-fixed}. If $G$ is of type
  $\sB_2$ or $\sG_2$ then $H$ is conjugate to either $L$ or
  $L'$ where $L$ is a Levi subgroup of a parabolic subgroup of $G$.

\end{lemma}

\begin{proof}

  By replacing $U_H$ by a conjugate, we may assume that $U_H\subset
  B^-$. Let $C$ be a component of $(G/P_\alpha)^{U_H}$ of codimension
  one. Then there are two possibilities: either $C$ meets the open
  $B^-$-orbit in
  $G/P_\alpha$ or not. We claim that in both cases $U_H$ is conjugate
  in $G$ to a root subgroup $U_{\gamma}$ for some simple root $\gamma$.
  Let $y\in C$ be general.

  In the first case, $B_y^-$ is conjugate to $TU_{-\alpha}$. Because
  of $U_H\subset B_y^-$ we conclude that $U_H$ is $G$-conjugate to
  $U_{\alpha}$, which implies the claim with $\gamma=\alpha$.

  In the second case, $C$ equals the Bruhat cell of codimension
  one. In that case $B_y^-$ is $B^-$-conjugate to $TU_0$ with
  $U_0=U_{-\beta}U_{-\gamma}$ where $\beta$ is the simple root
  different from $\alpha$ and $\gamma=s_\beta(\alpha)$. But now $U_H$
  is contained in every $B^-$-conjugate of $TU_0$, i.e., in the
  biggest connected subgroup $U_1$ of $U_0$ which is normal in
  $B^-$.

  Since $U_1$ is normalized by $T$, it is either trivial,
  $U_{-\beta}$, $U_{-\gamma}$, or $U_0$. It can't be trivial since it
  contains $U_H$. Moreover, a short case-by-case consideration shows
  that if $G=\mathsf G_2$ then none of the other three groups are
  normal in $B^-$.

Then $G$ is of type $\mathsf B_2$. Another short case-by-case
consideration shows that $U_1$ is either trivial (which is impossible
since $U_H\subseteq U_1$) or $U_{-\gamma}$ (where in this case $\gamma$
is long), proving the claim.

  Now we may assume that $U_H=U_{\gamma}$. Since both a maximal torus
  $S$ of $H$ and the maximal torus $T$ of $G$ are contained in the
  normalizer of $U_\gamma$ in $G$,
  we can replace $H$  by a conjugate such that $B\cap H$ contains both $S$
  and $U_\gamma$. The intersection $B\cap H$ is then a Borel subgroup of $H$. Let
  $L$ be the Levi subgroup corresponding to $\gamma$. Then $B\cap
  H\subseteq L\cap H$ is parabolic in $H$. Thus $H/L\cap H$ is a
  projective subvariety of the affine variety $G/L$. Thus $H\subseteq
  L$ proving the lemma.
\end{proof}

\begin{corollary}\label{cor:red}
Let $G$ be a semisimple group of rank $2$ and $H\subset G$ a connected
reductive subgroup of semisimple rank $1$. Then $H$ is spherical, or
$|\Bnn(G/H)|=1$, or $H=L$ or $L'$ where $L$ is a Levi subgroup of a
proper parabolic subgroup $P$ of $G$.
\end{corollary}

\begin{proof}
If the rank of $H$ is $2$ then $H$ is a Levi subgroup of a parabolic
subgroup of $G$.

If the rank of $H$ is $1$ then $H$ is of type $\sA_1$. If
$G=\sA_1\times\sA_1$ then $H$ is either one of the factors, in which
case $|\Bnn(G/H)|=1$, or $H$ is embedded
diagonally, possibly by a power of the Frobenius morphism. Thus, there
is an inseparable isogeny $\phi$ of $G$ such that $\phi(H)$ is the
diagonal subgroup of $G$. This shows that $H$ is spherical.

If $G=\sA_2$ then $H$ is embedded into $G$ via a $3$-dimensional
representation $V$. Because $p\ne2$, there are no non-trivial
extensions of a two dimensional representation $\Bbbk^2$ and the trivial
representation $\Bbbk$ (see e.g.\ \cite[Proposition~2.6]{St10}). Thus,
$V=\Bbbk^2\oplus \Bbbk$, or $V=\Bbbk^3$ is
irreducible. In the first case $H=\SL(2)\times 1\subset G=\SL(3)$, in
the second case $H=\SO(3)\subset G=\SL(3)$. In both cases $H$ is
spherical.

Ir remains to check the case $G=\sB_2$ or $G=\sG_2$ and $H$ is of
semisimple rank one. In this case the corollary follow from
Lemma~\ref{lemma:red2}.
\end{proof}

\section{Levi subgroups}\label{s:levi}

In this section we discuss subgroups $H$ of $G$ that are Levi subgroups of
some parabolic, up to a $\Bbbk^*$-factor.

\begin{proposition}\label{prop:levi}
Let $G$ be simple of rank $2$, $p\geq 3$, let $P\supsetneq B$ be a
proper parabolic subgroup of $G$, and let $H$ be a Levi subgroup of $P$.
If $G$ has type $\sA_1\times\sA_1$, $\mathsf A_2$ or $\mathsf B_2$ then
$H$ is spherical; if $G$ has type $\mathsf G_2$ then $H$ is spherical
or $|\Bnn(G/H)|=1$.
\end{proposition}
\begin{proof}
We may assume that $G$ is simply connected. If it has type
$\sA_1\times\sA_1$ then $H$ is spherical because a maximal torus of
$\SL(2)$ is spherical in $\SL(2)$ (see \cite[Theorem 4.3]{Br98}).

If $G$ has type $\mathsf A_2$ or $\mathsf B_2$ then $H$ itself appears
in \cite[Table 1]{Br98}, hence we apply again \cite[Theorem 4.3]{Br98}.


Suppose that $G$ has type $\mathsf G_2$, and consider the case
$P=P_{\alpha_1}$. Assume also $p\geq 5$: indeed, if $p = 3$ then
$P_{\alpha_1}$ is sent onto $P_{\alpha_2}$ by a special isogeny
of $G$, so we refer to our later discussion of the case
$P=P_{\alpha_2}$.

Thanks to the assumption $p\geq 5$ the module $V=V_G(\omega_1)$ is also
the Weyl module of $G$ associated to $\omega_1$ (see e.g.\
\cite[Theorem 1]{Pr88}), and $G/P_{\alpha_1}$
is a subvariety of $\PP(V)$. Let $U_H$ be a maximal unipotent subgroup
of $H$. The weights of the $G$-module $V$ imply that the latter is the sum of
three irreducible $H$-modules, therefore $(G/P_{\alpha_1})^{U_H}$ has
components of dimension at most $2$.

If $|\Bnn(G/H)|>1$ then Lemma~\ref{lemma:UH-fixed} applies: since $\dim
G/B=6$, then $(G/P_{\alpha_1})^{U_H}$ or $(G/P_{\alpha_2})^{U_H}$ has a
component $C$ of dimension $4$. The first case is excluded, so consider
the second case. The unipotent group $U_H$ has at least one fixed point
in $\pi_{\alpha_2}^{-1}(c)$ for all $c\in C$. It follows that
$(G/P_{\alpha_1})^{U_H}$ has a component of dimension at least $3$,
which is also impossible. Hence$|\Bnn(G/H)|=1$

It remains the case where $P=P_{\alpha_2}$ and $p\geq 3$. Let $U_H$ be
a maximal connected unipotent subgroup of $H$. We claim that the set of
$U_H$-fixed points on $G/P_{\alpha_i}$ has components of dimension at
most $3$ for both $i=1,2$: from Lemma~\ref{lemma:UH-fixed} we obtain $|
\Bnn(G/H)|=1$ as above.

To prove the claim, we use the commutation relations of
\cite[Expos\'e XXIII: Proposition 3.4.1(iii)]{SGA3} as in the proof of
Lemma~\ref{lemma:U}. To simplify notations, write
$\alpha=\alpha_1$ and $\beta=\alpha_2$.

Denoting $Q = P_{\alpha}$, we have
\[
G/Q = (Q^u w_0 Q/Q) \cup (Q^u s_{\beta} w_0 Q/Q) \cup \,
\textrm{(subvarieties of dimension $\leq 3$)}.
\]
Let us compute the $U_H$-fixed points on
$Q^u w_0 Q/Q$. We may assume that $U_H$ is the set of elements of the
form $u=u_{\beta}(x)$ for $x\in \Bbbk$.

If $v\in Q^u$, then $v w_0 Q\in Q^uw_0Q/Q$ is fixed under the action of
$u$ if and only if $v^{-1}uv\in {}^{w_0}Q$. Now write $v^{-1}$ as a
product:
\[
v^{-1} = u_{3\alpha+\beta}(y_1)u_{2\alpha+\beta}(y_2)u_{\alpha+\beta}
y_3)u_{3\alpha+2\beta}(y_4)u_{\beta}(y_5).
\]
Then:
\[
v^{-1} u_\beta(x) v = u_{3\alpha+2\beta}(-xy_1)u_\beta(x).
\]
This belongs to ${}^{w_0}Q$ only if $x=0$, therefore there are no
$U_H$-fixed points on $Q^u w_0 Q/Q$.

Let us do the same with $Q^u s_{\beta} w_0 Q/Q$, and call $w_1 =
s_\beta w_0$. This set is an affine space of dimension $4$, and we may
take $y_1,\ldots,y_4$ as its coordinates.

A point $v w_1 Q$ is fixed by $u$ if and only if $v^{-1} u v \in
{}^{w_1}Q$. This time, $u_{3\alpha+2\beta}(-xy_1)u_\beta(x)$ lies in
${}^{w_1}Q$ for all $x$ if and only if $y_1=0$. It follows that the set
of $U_H$-fixed points on $Q^u s_{\beta} w_0 Q/Q$ is irreducible of
dimension $3$.

Finally, we discuss the other parabolic $P=P_\beta$ with the same
procedure. Write
\[
G/P = (P^u w_0 P/P) \cup (P^u s_{\alpha} w_0 P/P) \cup \,
\textrm{(subvarieties of dimension $\leq 3$)},
\]
and consider $v\in P^u$. The point $v w_0 P$ is fixed by $u\in U_H$ if
and only if $v^{-1}uv\in{}^{w_0}P$. If:
\[
v^{-1} = u_{\alpha}(y_1)u_{3\alpha+\beta}(y_2)u_{3\alpha+2\beta}
y_3)u_{2\alpha+\beta}(y_4)u_{\alpha+\beta}(y_5)
\]
and $u=u_\beta(x)$, then:
\[
\begin{array}{ll}
v^{-1}uv & = \big(u_{\alpha}(-y_1)u_{3\alpha+\beta}(-y_2)\big)
u_\beta(x) \big(u_{3\alpha+\beta}(y_2)u_{\alpha}(y_1)\big)\\
& = u_\beta(x)u_{\alpha+\beta}(xy_1)u_{2\alpha+\beta}
(xy_1^2)u_{3\alpha+\beta}(xy_1^3)u_{3\alpha+2\beta}(x^2y_1^3-xy_2).
\end{array}
\]
It belongs to ${}^{w_0}P$ for all $x$ if and only if $y_1=y_2=0$, hence
the set of $U_H$-fixed points on $P^u w_0 P/P$ is $3$-dimensional. On
the other hand, for no $v\in P^u$ we have $v^{-1}uv\in {}^{s_\alpha
w_0}P$ for all $x$, therefore $U_H$ has no fixed points on $P^u
s_\alpha w_0 P/P$, and the proof is complete.
\end{proof}

\begin{corollary} \label{cor:levi}
Let $H \subset G$ be a connected subgroup. Suppose that $L'\subseteq H
\subseteq L$, where $L$ is a Levi subgroup of some proper parabolic
subgroup $P\supsetneq B$. Then Theorem~\ref{thm:waction} holds for $H$.
\end{corollary}
\begin{proof}
Suppose that $G$ has type $\mathsf A_1\times\mathsf A_1$ and is simply
connected, so $G=\SL(2)\times\SL(2)$. Without loss of generality
$P=\SL(2)\times B_{\SL(2)}$, where $B_{\SL(2)}$ is a Borel subgroup of
$\SL(2)$. It follows that $H=\SL(2)\times K$ where $K$ is a subgroup of
a maximal torus of $B_{\SL(2)}$. Therefore the reflection associated to
one of the simple roots of $G$ acts trivially on $\Bnn(G/H)$, and
Theorem~\ref{thm:waction} follows.

If $G$ is simple, apply Lemma~\ref{lemma:normtorus} to $H\subseteq L$
and then Proposition~\ref{prop:levi} to $L$: the corollary follows.
\end{proof}

\section{Other subgroups}\label{s:other}

In this section we finish the proof of Theorem~\ref{thm:waction},
discussing the remaining connected subgroups $H$ of $G$.

The first result regards the representation theory of $\SL(2)$ and
will be useful in subsequent proofs.

\begin{lemma}\label{lemma:sl2}
Let $V$ be a finite dimensional $\SL(2)$-module, and $R\subseteq V$ an
$\SL(2)$-stable additive subgroup. Suppose that one of the following
two conditions is satisfied:
\begin{enumerate}
\item\label{lemma:sl2:simple} The module $V$ is nontrivial and simple. 
\item\label{lemma:sl2:kpsimple} The characteristic $p$ of $\Bbbk$ is $\neq 2$,
the subgroup $R$ is closed and connected, and $V= V(0)\oplus V'$ where $V'$ is
nontrivial and simple.
\end{enumerate}
Then $R$ is a submodule of $V$.
\end{lemma}
\begin{proof}
We may suppose that $R\neq \{0\}$. Assume (\ref{lemma:sl2:simple}): then the union
of the sets $aR$ for all $a\in \Bbbk$ is a non-zero $\SL(2)$-submodule of
$V$, therefore equal to $V$. It follows that $R$ contains a highest
weight vector $v$, therefore all its multiples and all linear combinations of
elements of the form $gv$ for $g\in \SL(2)$. Hence $R=V$.

Assume now (\ref{lemma:sl2:kpsimple}), and let $R'$ be the projection of $R$ on
$V'$ along $V(0)$. Since it is an $\SL(2)$-stable additive subgroup of $V'$, as in
the proof of the first part of the Lemma we conclude that $R'$ is either $\{0\}$,
or contains a highest weight vector $v\in V'$.

In the first case $R$ is either $\{0\}$ or $V(0)$, since it is closed and connected.
In the second case, let $r\in R$ project to $v$: since $p\neq 2$ it is elementary to
show that $R$ also contains both
projections of $r$ in $V(0)$ and $V'$, and this completes the proof.
\end{proof}

\begin{lemma}\label{lemma:leviex}
Let $G$ be semisimple of rank $2$ and let $p\geq3$. Then any connected
subgroup $H$ of $G$ has a Levi subgroup. If $H$ is contained in a
parabolic subgroup $P$ of $G$, then any Levi subgroup of $H$ is
contained in a Levi subgroup of $P$.
\end{lemma}
\begin{proof}
The proposition is true if $H$ is solvable or very reductive, i.e.\ not
contained in any proper parabolic subgroup of $G$. Therefore we may
assume that $H$ is contained in a proper parabolic subgroup
$P\supsetneq B$ but not in any Borel subgroup of $G$. In this case
$HP^{u}/P^{u}$ is not contained in any proper parabolic
subgroup of $P/P^{u}$, which implies $H/H\cap P^u\cong HP^u/P^u$ reductive,
and hence $H^{u}=H \cap P^{u}$.

Denote by $L\supset T$ the standard Levi subgroup of $P$; according
to the decomposition $P = L \ltimes P^u$ we define two projections
$\pi_\ell\colon P\to L$ and $\pi_u\colon P\to P^u$. Notice that
$L$ is the quotient of $\SL(2) \times \GG_m$ by a finite central subgroup
scheme.

If $H$ contains a maximal torus of $G$ then it contains $T$ up to conjugation
by an element of $P$; we may then suppose that $H\supset T$. It follows
that $H$ contains $\pi_\ell(H)$, because
$\pi_\ell(h) \in \overline{ \{tht^{-1}\;|\; t\in T\} }$ for any $h\in H$.
This implies
that $H = \pi_\ell(H)\ltimes (H\cap P^{u})$, whence both
statements of the proposition.

We are left with the case where $H$ doesn't contain any maximal torus
of $G$, hence $\pi_\ell(H)\cong H/H^{u}$ is semisimple of rank $1$.
With this assumption, we show that the two statements of the
proposition follow from the vanishing of $H^1(L',R)$ and
$H^2(L',R)$ for certain subquotients $R$ of $P^{u}$. 

Consider the lower central series $P^u=P^{u}_0\supseteq P^u_1\supseteq
P^u_2 \supseteq\ldots$ of $P^u$, and the projection $\pi_1\colon
P\to P/P^u_1$. Then $\pi_1(H^{u})$ is a
$\pi_\ell(H)$-stable subgroup of $P^{u}/P^u_1$, which is a
vector group where $\pi_\ell(H)$ acts linearly by conjugation (see
\cite[Expos\'e XXVI: Proposition 2.1]{SGA3}).
If $H^2(\pi_\ell(H),\pi_1(H^{u}))=0$ then $\pi_1(H)$ is isomorphic to the
semidirect product of $\pi_\ell(H)$ and $\pi_1(H^u)$, thus it has a
Levi subgroup $L_1\subseteq P/P^u_1$. In particular
$L_1\cap\pi_1(H^{u})$ is trivial.

Let now $H_1=H\cap \pi_1^{-1}(L_1)$: this group has unipotent radical
contained in $P^u_1$, and satisfies $\pi_\ell(H)=\pi_\ell(H_1)$. We may
go on applying the same procedure to the group $H_1$ using the
projection onto the quotient $P/P^u_2$, provided that the corresponding
cohomology groups vanish. We obtain a sequence $H\supseteq H_1\supseteq
H_2\supseteq\ldots$ of subgroups of $H$ satisfying
$\pi_\ell(H)=\pi_\ell(H_i)$ and $H_i^{u}\subseteq P^u_i$ for all $i$.
If $n$ is big enough so that $P^u_n$ is trivial, the subgroup $H_n$ of
$H$ is reductive and isomorphic to $\pi_\ell(H)$, hence it is a Levi
subgroup of $H$.

Denote now by $L_H$ a Levi subgroup of $H$, and consider the map
$(\pi_1\circ\pi_u)|_{L_H}\colon L_H\to P^u/P^{u}_1$. It is a
$1$-cocycle of $L_H$ with values in the module $P^u/P^{u}_1$.
If the group $H^{1}(L_H,P^u/P^{u}_1)$ vanishes, then
$(\pi_1\circ\pi_u)|_{L_H}$ is a coboundary, whence $\pi_1(L_H)$ is
contained in $\pi_1(L)$ up to conjugation by an element of
$P^u/P^{u}_1$.

Therefore we may assume that $L_H\subset LP^{u}_1$. Proceeding as
above using the projections on $P/P^{u}_i$ for $i\in\{2,3,\ldots\}$,
provided that the needed
cohomology groups vanish, we obtain that $L_H$ is contained in $L$ up
to conjugation by an element of $P^{u}$.

To finish the proof we must show the vanishing of the cohomology groups
involved. We notice that for all of them we may replace the group with
its image under $\pi_\ell$, i.e.\ with $L'$. Then it remains
to show that $H^{n}(L',R)=0$, where
$n\in\{1,2\}$ and $R$ is an $L'$-stable additive subgroup
of $P^{u}_{j}/P^{u}_{j+1}$.
Using the long exact sequence of group cohomology the problem is
reduced to the case where $R$ is an $L'$-stable
subgroup of a simple subquotient $Q$ of
$P^{u}_{j}/P^{u}_{j+1}$.

If $Q$ is the trivial simple $L'$-module
then $L'$ acts trivially on $R$; this implies $H^1(L', R)=0$ because
$\SL(2)$ has no non-trivial abelian quotient.
In this case also $H^2(L',R)=0$ follows, using the long exact sequence of group
cohomology, the vanishing of $H^1(L', Q/R)$ by the above argument applied to
$Q/R$ instead of $R$, and
the vanishing of $H^2(L', Q)$ (see \cite[Theorem 1]{St10}).
We may now assume that $Q$ is not the trivial simple $L'$-module, and
that $R\neq\{0\}$.

From Lemma~\ref{lemma:sl2}, part (\ref{lemma:sl2:simple}) it follows that
$R=Q$. Moreover, by inspection on all rank $2$
root systems, the $L'$-module $Q$ (viewed as an $\SL(2)$-module)
is isomorphic to
$V(\omega_1)$, $V(2\omega_1)$ or $V(3\omega_1)$. Finally,
the vanishing of $H^{1}(L',Q)$ follows from 
\cite[Proposition~2.6]{St10}, and the vanishing of $H^{2}(L',Q)$
follows from \cite[Theorem 1]{St10}. Notice that here is the step
where we use the assumption $p> 2$, since e.g.\ $H^{1}(\SL(2),
V(\omega_1))\neq 0$ in characteristic $2$.
\end{proof}

\begin{lemma}\label{lemma:openorbits}
If $H\subseteq G$ has an open orbit on $G/P_{\alpha}$ for each simple
root $\alpha$ but not on $G/B$, then $|\Bnn(G/H)|=1$.
\end{lemma}
\begin{proof}
If $|\Bnn(G/H)|>1$ then Proposition~\ref{prop:transitive} implies the
existence of an element $Y\in \Bnn(G/H)$ of codimension $1$ in $G/H$
satisfying  $s_{\alpha}\cdot(G/H)=(Y)$ for some simple root $\alpha$.
Since $H$ has no open orbit on $G/B$, we have $c(Y)=c(G/H)>0$.
Then $s_{\alpha}\cdot(G/H)=(Y)$ implies that the generic $H$-sheet of
$G/P_{\alpha}$ has positive complexity, which contradicts our assumptions.
\end{proof}

\begin{lemma}\label{lemma:sphercrit}
Let $P$ and $P_-$ be two opposite parabolic subgroups of $G$, set
$L=P\cap P_-$ and let $I$ be either $L$ or $L'$. Let also $H\subseteq P$
be a connected subgroup containing $I$. Then $H^u\subseteq P^u$; if
$P^u/H^u$ has an open $I$-orbit then $H$ has an open orbit on $G/P_-$,
and if $P^u/H^u$ is spherical under the action of $I$ then $H$ is a
spherical subgroup of $G$.
\item 
\end{lemma}
\begin{proof}
The inclusion $H^u\subseteq P^u$ stems from the inclusion
$H^u\subseteq P^r$, which holds because $H P^r/P^r=P/P^r$
is reductive, therefore $H^u$ is in the kernel of the
projection $P\to P/P^r$.

Consider a Borel subgroup $B_L\subseteq L$. Then $B_-=B_L P_-^u$ is a Borel
subgroup of $G$, and its subgroup $P_-^u$ has an open orbit on $G/P$.
Then $L$ has an open orbit on $P/H$ if and only if $P_-$ has an
open orbit on $G/H$, and $B_L$ has an open orbit on $P/H$ if and
only if $B_-$ has an open orbit on $G/H$.

This completes the proof if $I=L$, because then $P/H=P^u/H^u$; in this
case the two last statements of the lemma are even equivalences. If
$I=L'$ but $H$ anyway contains $L$ then again $P/H=P^u/H^u$ and the two
last statements of the lemma follow from the case $I=L$. Hence we can
suppose $H\nsupseteq L$.

In this case $P/H=P^r/H^u$: if $L'$ has an open orbit on $P^u/H^u$
then $L$ has an open orbit on $P^r/H^u$ thus $P_-$ has an open orbit on $G/H$,
and if $P^u/H^u$ is spherical under the action of $L'$ then $P^r/H^u$ is
spherical under the action of $L$ thus $H$ is a spherical subgroup of $G$.
\end{proof}

\begin{proposition}\label{prop:other}
Let $G$ be semisimple of rank $2$ and $p\geq 3$. Let $P\supsetneq B$ be
a proper parabolic subgroup of $G$ and $H$ a connected non-reductive
subgroup of $P$ containing a Levi subgroup $L$ of $P$. Then $H$ is
spherical or $|\Bnn(G/H)|=1$.
\end{proposition}
\begin{proof}
If $G$ has type $\mathsf A_1\times \mathsf A_1$, $\mathsf A_2$ or
$\mathsf B_2$ then $L$ is spherical (Proposition~\ref{prop:levi}). This
implies that $H$ is also spherical.

So we suppose that $G$ has type $\mathsf G_2$, and also that $L\supset
T$. Write for brevity $\alpha=\alpha_1$ and $\beta=\alpha_2$.

Since $H\supseteq L$, $P$ is
minimal among the parabolic subgroups containing $H$: it follows that
$H^u\subseteq P^u$. The quotient $P^{u}/(P^{u})'$ is a vector group
where $L$ acts linearly, and $H^{u}(P^{u})'/(P^{u})'$ is an additive
subgroup. Since the center of $L$ acts on $P^{u}/(P^{u})'$ non-trivially
by homotheties and
$H^{u}$ is $L$-stable, the group $H^{u}(P^{u})'/(P^{u})'$ is an
$L$-submodule of $P^{u}/(P^{u})'$. 

Moreover, we may suppose that it is a proper submodule, otherwise
$H^{u}$ is the whole $P^{u}$, and hence $H=P$ is spherical. Notice that
here $(P^{u})'$ is abelian, so also a vector group where $L$ acts
linearly. Let us then recall the structure of the
$L$-modules under consideration (viewed as $\SL(2)\times \GG_m$-modules): 
\[
\begin{array}{rcl}
P_\alpha^u/(P_\alpha^{u})' & \cong & \Sym^3(V(\omega_1))\otimes
V(\beta|_{\GG_m}),\\
(P_\alpha^{u})' & \cong & V(0)\otimes V(2\beta|_{\GG_m}),
\end{array}
\]
and
\[
\begin{array}{rcl}
P_\beta^u/(P_\beta^{u})' & \cong & V(\omega_1)\otimes
V(\alpha|_{\GG_m}),\\
(P_\beta^{u})' & \cong & \big(V(0)\otimes V(2\alpha|_{\GG_m})\big) \oplus
\big(V(\omega_1)\otimes V(3\alpha|_{\GG_m})\big).
\end{array}
\]

If now $p>3$ or if $P=P_\beta$ then $H^{u}(P^{u})'/(P^{u})'$ is
trivial, because with either of these two assumptions $P^{u}/(P^{u})'$
is a simple $L$-module. For $p=3$ and $P=P_\alpha$, the
$L$- module $P^{u}/(P^{u})'$ contains a unique nontrivial and proper
$L$-submodule (of dimension $2$). However, if $H^{u}(P^{u})'/(P^{u})'$
contains this submodule then $P^u/H^u$ is spherical under the action of
$L$, therefore $H$ is spherical thanks to Lemma~\ref{lemma:sphercrit}.

As a consequence, we may suppose from now on that $H^u\subseteq (P^u)'$.
Now $(P^{u})'$ as an $\SL(2)$-module it is either $V(0)$ or
$V(0)\oplus V(\omega_1)$, and we observe that in both cases $H^u$ is an
$L$-submodule. This is obvious in the first case since
$H^u$ is connected, and in the second case it follows from
Lemma~\ref{lemma:sl2}, part (\ref{lemma:sl2:kpsimple}).

In addition, if $P=P_\beta$
and $H^u$ contains the $2$-dimensional $L$-submodule of $(P^u)'$, then again
$P^u/H^u$ is $L$-spherical, so $H$ is spherical in $G$.

This leaves only one subgroup $H$ for each of the two choices of $P$,
namely the one
where $H^u$ is equal to the $1$-dimensional summand of $(P^u)'$. We claim
that in both cases $H$ has an open orbit on $G/P_\alpha$ and $G/P_\beta$.
This proves the proposition thanks to Lemma~\ref{lemma:openorbits}, since
$H$ has no open orbit on $G/B$ for dimension reasons.

Let us prove the claim for $P=P_\alpha$, and consider first
$G/P_\alpha$. Thanks to Lemma~\ref{lemma:sphercrit}, to prove that
$H\subset P_\alpha$ has an open orbit on $G/P_\alpha$ it is
enough to prove that $L$ has an open orbit on $P_\alpha^u/H^u$.
Both $L$ and $P_\alpha^u/H^u\cong
\Sym^3(V(\omega_1))\otimes V(\beta|_{\GG_m})$ have dimension $4$, so
our claim follows from the fact that the point $e_1 e_2 (e_1+e_2)$ has
finite stabilizer in $L$, where $e_1,e_2$ is the standard basis of
$V(\omega_1)\cong\Bbbk^{2}$.

We show now that $H\subset P_\alpha$ has an open orbit on $G/P_\beta$.
Notice that a Levi subgroup of $P_\beta$ and a Levi subgroup of $P_\alpha$
are not conjugated in $G$, whence the group $({}^{g}P_\beta \cap
H)^{\circ}$ is solvable for all $g\in G$. It is then contained in a Borel
subgroup of $H$. Since the flag variety of $H$ has dimension $1$ and
$B\cap H$ is a Borel of $H$, the inclusion
\[
({}^{g}P_\beta \cap H)^{\circ} \subseteq B\cap H
\]
holds for all $g$ such that $gP_\beta\in D$, where $D$ is a subvariety of
$G/P_\beta$ of codimension $1$. If $H$ has no open orbit on $G/P_\beta$
then $({}^{g}P_\beta \cap H)$
has positive dimension for all $g\in G$, therefore for all $g$ such that
$gP_\beta\in D$ we have
\begin{equation}\label{eqn:posdim}
\dim({}^{g}P_\beta\cap B \cap H) >0.
\end{equation}
We prove that this is impossible, by checking that that the intersection
of the locus where (\ref{eqn:posdim}) is satisfied with any Schubert cell
of $G/P_\beta$ has codimension at least $2$ in $G/P_\beta$.

We first consider the open Schubert cell, i.e.\ $g$ is
of the form $g=uw_0$ with $u\in U$. Then
\[
({}^{g}P_\beta \cap B) \cap H = ({}^{uw_0}P_\beta \cap B)\cap H =
{}^{u}(T U_\beta)\cap H
\]
We fix an isomorphism of the open Schubert cell with $\Bbbk^{5}$
in such a way that 
\[
u = u_{\alpha+\beta}(x_1)u_{2\alpha+\beta}(x_2)u_{3\alpha+\beta}(x_3)
\]
where the parameters $x_1$, $x_2$ and $x_3$ are coordinate
functions of $\Bbbk^{5}$. This is possible since $U_\alpha
U_{3\alpha+2\beta}\subset H$.
An elementary computation shows that ${}^{u}(T 
U_\beta)\cap H$ is infinite only if two of $x_1,x_2,x_3$ are zero.

Consider now the codimension $1$ Schubert cell, i.e.\ $g$ is of the form
$uw_0s_\alpha$. Then
\[
({}^{g}P_\beta \cap B) \cap H = ({}^{uw_0s_\alpha}P_\beta \cap B)\cap H =
{}^{u}(T U_\alpha U_{3\alpha+\beta})\cap H
\]
Here we can assume that
\[
u = u_{\beta}(x_1)u_{\alpha+\beta}(x_2)u_{2\alpha+\beta}(x_3)
\]
where $x_1,x_2,x_3$ are coordinate functions on the $4$-dimensional
Schubert cell. Consider also elements $t\in T$,
$u_{\alpha}(y_1)$ and $u_{3\alpha+\beta}(y_2)$ with $y_1,y_2
\in \Bbbk$. Then ${}^{u}(tu_\alpha(y_1)u_{3\alpha+\beta}(y_2))$
is equal to
\[
t
u_{\beta}(A_\beta)
u_{\alpha+\beta}(A_{\alpha+\beta})
u_{2\alpha+\beta}(A_{2\alpha+\beta})
u_{3\alpha+\beta}(A_{3\alpha+\beta})
u_{3\alpha+2\beta}(A_{3\alpha+2\beta})
u_{\alpha}(A_\alpha)
\]
where
\[
\begin{array}{lcl}
A_\beta &=& (b - 1)x_1\\
A_{\alpha+\beta} &=& (ab-1)x_2 + x_1y_1\\
A_{2\alpha+\beta} &=& (a^{2}b-1)x_3+2x_2y_1-x_1y_1^{2}\\
A_{3\alpha+\beta} &=& y_2-6x_2y_1^{2}+x_1y_1^{3}
\end{array}
\]
and $a = \alpha(t^{-1})$, $b = \beta(t^{-1})$. This element is in $H$ only
if $A_\beta = A_{\alpha+\beta} = A_{2\alpha+\beta} = A_{3\alpha+\beta}
=0$. For general $x_1$, $x_2$ and $x_3$ the resulting system has only
finitely many solutions in $a,b,y_1,y_2$. This finishes the proof that 
(\ref{eqn:posdim}) cannot be satisfied for $gP_\beta$ lying on a 
subvariety of $G/P_\beta$ of codimension $1$.

Finally, we prove the claim for $P=P_\beta$ using the same
method. Here $H\subset P_\beta$, and we consider first $G/P_\beta$.
The quotient $P_\beta^u/H^u$ is abelian and, as an
$\SL(2)\times \GG_m$-module, isomorphic to the sum
\[
(V(\omega_1)\otimes V(\alpha|_{\GG_m})) \oplus
\big(V(\omega_1)\otimes V(3\alpha|_{\GG_m})\big) \cong
\Bbbk^{2}\oplus\Bbbk^{2}.
\]
The group $L$ has the open orbit 
\[
\{(v,w)\in \Bbbk^2\oplus \Bbbk^2\,|\, v\neq0\neq w,\,
\Bbbk v\neq \Bbbk w\},
\]
hence $H$ has an open orbit on $G/P_\beta$.

To show that $H\subset P_\beta$ has an open orbit on $G/P_\alpha$
we proceed as above, showing that the locus where $gP_\alpha$
satisfies
\[
\dim({}^{g}P_\alpha\cap B \cap H) >0
\]
has codimension at least $2$ in $G/P_\alpha$.

Suppose first that $gP_\alpha$ is in the open Schubert cell,
i.e.\ $g$ is of the form $uw_0$ for $u\in U$. Then
\[
({}^{g}P_\alpha \cap B) \cap H = ({}^{uw_0}P_\alpha \cap B)\cap H
= {}^{u}(T U_\alpha)\cap H
\]
We can write
\[
u = u_{\alpha+\beta}(x_1)u_{3\alpha+\beta}(x_2)u_{3\alpha+2\beta}(x_3)
\]
since here $U_\beta U_{2\alpha+\beta}\subset H$. Again, the intersection
${}^{u}(T U_\alpha)\cap H$ is infinite only if two
of $x_1$, $x_2$, $x_3$ are zero.

Consider now the codimension $1$ Schubert cell, i.e.\ $g$ is of
the form $uw_0s_\beta$. Then
\[
({}^{g}P_\alpha \cap B) \cap H = ({}^{uw_0s_\beta}P_\beta \cap B)\cap H =
{}^{u}(T U_\beta U_{\alpha+\beta})\cap H
\]
Here we can assume that
\[
u = u_{\alpha}(x_1)u_{3\alpha+\beta}(x_2)u_{3\alpha+2\beta}(x_3)
\]
where $x_1,x_2,x_3$ are coordinate functions on the $4$-dimensional
 Schubert cell.
Consider also elements $t\in T$,
$u_{\beta}(y_1)$ and $u_{\alpha+\beta}(y_2)$ with $y_1,y_2
\in \Bbbk$.
Then ${}^{u}(tu_\beta(y_1)u_{\alpha+\beta}(y_2))$ is equal to
\[
t
u_{\alpha}(A_\alpha)
u_{3\alpha+\beta}(A_{3\alpha+\beta})
u_{3\alpha+2\beta}(A_{3\alpha+2\beta})
u_{2\alpha+\beta}(A_{2\alpha+\beta})
u_{\beta}(A_\beta)
u_{\alpha+\beta}(A_{\alpha+\beta})
\]
where
\[
\begin{array}{lcl}
A_\alpha &=& (a-1)x_1 \\
A_{3\alpha+\beta} &=& (a^{3}b -1) x_2 + x_1^{3}y_1 - 3x_1^{2}y_2\\
A_{3\alpha+2\beta} &=& (a^{3}b^{2}-1)x_3 -x_2y_1 -
9x_1^{2}y_1y_2 - x_1^{3}y_1^{2}-3x_1y_2^{2}\\
A_{\alpha+\beta} &=& x_1y_1 + y_2
\end{array}
\]
and $a = \alpha(t^{-1})$, $b = \beta(t^{-1})$.
This element is in
$H$ only
if $A_\alpha = A_{3\alpha+\beta} = A_{3\alpha+2\beta} =
A_{\alpha+\beta} =0$.
For general $x_1$, $x_2$ and $x_3$ the resulting system has
finitely many solutions in $a,b,y_1,y_2$. This finishes the proof.
\end{proof}

\begin{corollary}\label{cor:other}
Theorem~\ref{thm:waction} holds for every connected subgroup $H$ 
of $G$ such that $H$ is neither solvable nor very reductive, and
does not satisfy $L'\subseteq H\subseteq L$ where $P\supsetneq B$
is a proper parabolic subgroup of $G$ and $L$ is a Levi of $P$.
\end{corollary}
\begin{proof}
Since $H$ is not solvable nor very reductive, it is contained in a
proper parabolic subgroup $P$, which may be assumed to properly
contain $B$. Denote by $L$ a Levi subgroup of $P$. Thanks to
Lemma~\ref{lemma:leviex}, $H$ has a Levi subgroup
inside $L$ and containing $L'$. Moreover, the inclusion
$H^{u}\subseteq P^{u}$ holds thanks to Lemma~\ref{lemma:sphercrit},
and $H^u$ is not trivial thanks to our assumptions.

To prove the corollary it is enough to show that $H$ is spherical
or $L^r$ normalizes $H$. Indeed, if $L^r$ normalizes $H$, then by
Lemma~\ref{lemma:normtorus} the sets $\Bnn(G/H)$ and $\Bnn(G/L^rH)$
are $\Wt$-equivariantly isomorphic. In this case the corollary stems
from Proposition~\ref{prop:other}, since $L^rH$ contains $L$.

Suppose first that $P^{u}$ is abelian and a trivial or
simple $L'$-module, which is the case if $G$ has type
$\mathsf A_1\times\mathsf A_1$ or $\mathsf A_2$, or $G$ has
type $\mathsf B_2$ and $P=P_{\alpha_2}$.
Then by Lemma~\ref{lemma:sl2}, part (\ref{lemma:sl2:simple}), we
conclude that $H^u$ is a $L$-stable submodule of $P^u$, which implies
that $H$ is normalized by $L^r$.

If $G$ has type $B_2$ and $P=P_{\alpha_1}$, then $P^u/(P^u)'$ under the
action of $L'$ is the
$\SL(2)$-module $V(\omega_1)$ and the image $K$ of $H^u$ in
the quotient $P^u/(P^u)'$ is an $L'$-stable additive subgroup.
Lemma~\ref{lemma:sl2}, part
(\ref{lemma:sl2:simple}) implies that $K$ is
either the whole $V(\omega_1)$ or trivial. In the first case then
$H^u=P^u$. In the second case $H^u\subseteq (P^u)'$, and since the
former is non-trivial while the latter is one-dimensional we have
$H^u=(P^u)'$. In any case $H$ is both spherical and normalized by
$L^r$, and this concludes the proof in the case $G$ is not of
type $\mathsf G_2$.

We may now suppose that $G$ has type $\mathsf G_2$.
We start with the case $P=P_{\alpha_1}$.

If $p=3$ then $P^u$ is abelian, therefore $\SL(2)$-isomorphic
to $\Sym^3(V(\omega_1))\oplus V(0)$. The first summand has a
composition series $\Sym^3(V(\omega_1))\supseteq V_1 \supseteq \{0\}$ of
$\SL(2)$-submodules with both $V_1$ and $V/V_1$ simple
and $2$-dimensional.

Consider the projection $K$ of $H^{u}$ in $\Sym^3(V(\omega_1))$.
By Lemma~\ref{lemma:sl2}, part (\ref{lemma:sl2:simple}), the
intersection $K\cap V_1$ and the projection of $K$ on $V/V_1$ are
$\SL(2)$-submodules, therefore either $0$- or $2$-dimensional.

If both are $0$-dimensional then $H^{u}$ is nontrivial and
contained in $V(0)$, therefore $H^u = V(0)$. This implies that $H$ is
normalized
by $L^{r}$. Suppose that at least one is $2$-dimensional, and consider
the projection $J$ of $H^u$ on $V(0)\subset P^u$. Then either $J$ is
$V(0)$, which implies that $P^u/H^u$ is $\SL(2)$-spherical and
thus $H$ is a spherical subgroup of $G$, or $J$ is trivial,
which implies that $L^r$ normalizes $H$.

If $p>3$ then
$V=P^u/(P^u)'\cong\Sym^3(V(\omega_1))$ is irreducible under the
action of $\SL(2)$. It follows from Lemma~\ref{lemma:sl2}, part
(\ref{lemma:sl2:simple}) that the projection
of $H^{u}$ on $V$ is an $\SL(2)$-submodule, so either trivial,
and hence $H$ is normalized by $L^r$, or the full $V$, which
implies that $H^{u}=P^{u}$ and that $H$ is spherical.

We deal now with the case $P=P_{\alpha_2}$.
If $H^u\subseteq (P^u)'$, since the latter is abelian and the
sum of a trivial and a $2$-dimensional simple $\SL(2)$-module,
we conclude from Lemma~\ref{lemma:sl2}, part (\ref{lemma:sl2:kpsimple})
that $H^{u}$ is an $\SL(2)$-submodule, thus normalized by $L^r$.

Otherwise $H^u$ projects surjectively on $P^u/(P^u)'$, because the latter
is a simple $\SL(2)$-module. Then $H^{u}=P^{u}$ and $H$ is a spherical
subgroup of $G$.
\end{proof}

\end{document}